\documentclass[reqno]{amsart}

\usepackage{amsfonts,amssymb,amsmath}
\usepackage{hyperref}

\DeclareMathOperator{\gop}{\mathbf g}

\theoremstyle{plain}

\newtheorem{proposition}{Proposition}
\newtheorem{lemma}{Lemma}
\newtheorem{corollary}{Corollary}

\theoremstyle{definition}

\title{Fractoconvex structures}
\author{Aidar~Dulliev}
\address{A.M.~Dulliev: Kazan National Research Technical University named after A.N. Tupolev-KAI (Russia)}
\email{dulliev@yandex.ru}

\begin{document}

\begin{abstract}
We define a new structure on a space endowed with convexities, and call it a fractoconvex structure (or, a space with fractoconvexity). We introduce two operations on a set of fractoconvexities and in a special case we show that they satisfy the laws for a distributive lattice. We establish a connection between fractoconvex sets and convex sets using the concept of independent convexities, based on the possibility of representing a fractoconvex set as the intersection of its convex hulls. Finally, we consider some examples of fractoconvexities on the 2-sphere and on $\mathbb Z$.
\end{abstract}

\maketitle

\section{Introduction}
\label{sec1}
The concept of a convexity plays an important role in many topics of mathematics. In each of them the properties associated with convexity appear on an appropriate abstract level. The theory that deals with convexity and its applications from a general point of view was formed in the 1960--1980s and was called the theory of convex structures. A fairly complete exposition of the theory of convex structures can be found in the monographs \cite{1,2,3}. Following \cite{1}, we recall the most general viewpoint of the notion of convexity.

Let $X$ be a set. A {\itshape convexity} $\mathbf G\subset 2^X$ on $X$ is defined, as a rule, in two equivalent ways. In the first way, a collection $\mathbf G\subset 2^X$ is called a convexity on $X$ if $X\in\mathbf G$ and $\mathbf G$ is closed under arbitrary intersections of its elements. In the second way, a convexity is generated by a {\itshape convex hull operator} $\gop\colon 2^X\to 2^X$ such that, for any sets $A$, $B$ satisfying the inclusions $A\subset B\subset X$, the following conditions hold: $A\subset\gop A$, $\gop A\subset\gop B$, and $\gop\gop A=\gop A$; this convexity is in turn equal to $\mathbf G=\{A\subset X:\gop A=A\}$. A convexity $\mathbf G$ is called {\itshape finitely defined} if for every $A\subset X$ it follows that $\gop A=\bigcup\{\gop B: B\subset A,|B|<\infty\}$. A convexity $\mathbf G$ is called {\itshape $n$-ary} if $A\in\mathbf G\Leftrightarrow\forall B\subset A\bigl(|B|\leqslant n\Rightarrow\gop B\subset A\bigr)$ whenever $A\subset X$.

Now let the set $X$ be endowed with a family of convexities on it. Clearly, in $X$ we can construct new structures based on the convexities, and these structures are not necessarily convexities on $X$, but may have similarities with them. For example, in \cite{4}  the author proposed and investigated the notions of an $n$-semiconvex set and an $n$-biconvex set. We briefly recall these notions.

Let $\mathbf G_1$ and $\mathbf G_2$ be convexities on a set $X$, $\gop_1$ and $\gop_2$ be the convex hulls associated with them, respectively, and let $n$ be a natural number. A set $A\subset X$ is called {\itshape $n$-semiconvex} with respect to $\mathbf G_1$ and $\mathbf G_2$ if
\begin{equation}\label{1.1}
\forall B\subset A\bigl(|B|\leqslant n\Rightarrow\exists i\in\{1,2\}\;\gop_i B\subset A\bigr).
\end{equation}
A set $A\subset X$ is called {\itshape $n$-biconvex} with respect to $\mathbf G_1$ and $\mathbf G_2$ if
\begin{equation}\label{1.2}
\forall B\subset A\bigl(|B|\leqslant n\Rightarrow\forall i\in\{1,2\}\;\gop_i B\subset A\bigr).
\end{equation}

From (\ref{1.1}) and (\ref{1.2}) (see the quantifiers in the round brackets) it follows that, for fixed $n$, $\mathbf G_1$, and $\mathbf G_2$, the family of all $n$-semiconvex sets is not stable for intersections, but the family of all $n$-biconvex sets is in turn stable for intersections. Therefore, the family of all $n$-biconvex sets is a convexity on $X$. At the same time, it has been shown in \cite{4} that
\begin{itemize}
\item Some special $n$-semiconvex sets can be represented as the intersection of their convex hulls.
\item In some cases the theorems which are similar to the hyperplane separation theorems in $\mathbb R^n$ can be applied to $n$-semi- and $n$-biconvex sets.
\end{itemize}

In this paper the notions of $n$-semi- and $n$-biconvex set are generalized for an arbitrary, not necessarily finite, family of convexities. We define the notions of ``fractoconvexity'' (fractional convexity) and ``multiconvexity'' and investigate their properties.  We shall briefly consider the question of how fractoconvex sets can be represented by convex sets belonging to given convexities. Finally, we provide four examples to illustrate the notions and statements proposed in our paper.

\section{Fractoconvexities}
\label{sec2}
Let $\Lambda\ne\emptyset$ be an index set, $\mathbf G_\lambda$, $\lambda\in\Lambda$ be convexities on a set $X$, $\{M_i,|M_i|\geqslant 1\}$ be a partition of $\Lambda$ such that
\begin{gather*}
\forall i\bigl(\lambda_{1,2}\in M_i,\lambda_1\ne\lambda_2\Rightarrow\mathbf G_{\lambda_1}\ne\mathbf G_{\lambda_2}\bigr).
\end{gather*}
and, for every $i$, let $m_i$ be a cardinal number such that $1\leqslant m_i\leqslant |M_i|$. (Here and in what follows, by $|\cdot|$ we denote the cardinality of a set).

We say that a set $A$ is an {\itshape $n$-ary fractoconvex set of the type $\{(M_i,m_i)\}$ in $X$ with respect to the convexities $\mathbf G_\lambda$} (briefly, {\itshape $(n)\text{-}\bigvee\limits_i\dfrac{m_i}{\{\mathbf G_\lambda, \lambda\in M_i\}}$\--fractoconvex set in $X$}), if
\begin{gather*}
\forall B\subset A\Bigl(|B|\leqslant n\Rightarrow\bigl(\exists i\exists\Omega\subset M_i, |\Omega|=m_i: \bigcup\limits_{\lambda\in\Omega}\gop_\lambda B\subset A\bigr)\Bigr).
\end{gather*}
Similarly, we say that a set $A$ is an {\itshape $n$-ary multiconvex set in $X$ with respect to the convexities $\mathbf G_\lambda$} (briefly, {\itshape $(n)\text{-}\{\mathbf G_\lambda, \lambda\in \Lambda\}$\--multiconvex set in $X$}), if
\begin{gather*}
\forall B\subset A\Bigl(|B|\leqslant n\Rightarrow\bigcup\limits_{\lambda\in\Lambda}\gop_\lambda B\subset A\Bigr).
\end{gather*}

The collection of all $(n)\text{-}\bigvee\limits_i\dfrac{m_i}{\{\mathbf G_\lambda, \lambda\in M_i\}}$-fractoconvex (resp., $(n)\text{-}\{\mathbf G_\lambda, \lambda\in \Lambda\}$\--multiconvex) sets in $X$ will be called a {\itshape fractoconvexity} (resp., {\itshape multiconvexity}) and will be denoted by $(n)\text{-}\bigvee\limits_i\dfrac{m_i}{\{\mathbf G_\lambda, \lambda\in M_i\}}$ (resp., $(n)\text{-}\{\mathbf G_\lambda, \lambda\in \Lambda\}$). The pair $\left(X,(n)\text{-}\bigvee\limits_i\dfrac{m_i}{\{\mathbf G_\lambda, \lambda\in M_i\}}\right)$ will be called a {\itshape fractoconvex structure} (or, a {\itshape space with fractoconvexity}).

From these definitions we obtain that a multiconvexity is a particular case of a fractoconvexity, i.e. any $(n)\text{-}\{\mathbf G_\lambda, \lambda\in \Lambda\}$-multiconvex set is an $(n)\text{-}\dfrac{|\Lambda|}{\{\mathbf G_\lambda, \lambda\in\Lambda\}}$-fractoconvex set. Also, a set that is $n$-semiconvex with respect to $\mathbf G_1$ and $\mathbf G_2$ is an $(n)\text{-}\dfrac{1}{\{\mathbf G_1, \mathbf G_2\}}$\--fractoconvex set, a set that is $n$-biconvex with respect to $\mathbf G_1$ and $\mathbf G_2$ is an $(n)\text{-}\dfrac{2}{\{\mathbf G_1, \mathbf G_2\}}$\--fractoconvex set and is an $(n)\text{-}\{\mathbf G_1, \mathbf G_2\}$\--multiconvex set, a set that is convex with respect to an $n$-ary convexity $\mathbf G_1$ is an $(n)\text{-}\dfrac{1}{\{\mathbf G_1\}}$-fractoconvex set and is an $(n)\text{-}\{\mathbf G_1\}$-multiconvex set.

Let $\mathfrak G^{(\,n)}(X)$ be the collection of all $n$-ary convexities on $X$. By $\mathfrak F^{(\,n)}(X)$ (resp., $\mathfrak F^{(\,n)}_{\mathrm{fin}}(X)$) denote the collection of all $(n)$-fractoconvexities on $X$ with respect to all (resp., all finite) subsets of $\mathfrak G^{(\,n)}(X)$. It is obvious that $\mathfrak F^{(\,n)}_{\mathrm{fin}}(X)\subset\mathfrak F^{(\,n)}(X)$ and $\mathfrak F^{(\,n-1)}(X)\subset\mathfrak F^{(\,n)}(X)$. If not otherwise stated, {\slshape we assume that all convexities $\mathbf G_\lambda$, $\lambda\in\Lambda$ considered below are $n$-ary; additionally, the prefix $(n)$- for the fracto- and multiconvexities will be omitted}.

Take arbitrary fractoconvexities $\mathbf F_j=\bigvee\limits_i\dfrac{m^j_i}{\{\mathbf G^j_\lambda, \lambda\in M^j_i\}}\in\mathfrak F^{(\,n)}(X)$, $j\in\{1,2\}$. According to the definition of a fractoconvexity, we can also define a new fractoconvexity $\mathbf F_1\vee\mathbf F_2\in\mathfrak F^{(\,n)}(X)$ and a corresponding {\itshape operation $\vee\colon\mathfrak F^{(\,n)}(X)\times\mathfrak F^{(\,n)}(X)\to\mathfrak F^{(\,n)}(X)$} as follows:
\begin{gather*}
A\in\mathbf F_1\vee\mathbf F_2\Leftrightarrow\forall B\subset A\Bigl(|B|\leqslant n\Rightarrow\!\bigl(\exists j\in\{1,2\}\exists i\exists\Omega\subset M^j_i, |\Omega|=m^j_i: \bigcup\limits_{\lambda\in\Omega}\gop_\lambda B\subset A\bigr)\Bigr).
\end{gather*}

It is readily seen that the operation $\vee$ on the set $\mathfrak F^{(\,n)}_{\mathrm{fin}}(X)$ possesses the following properties (hereinafter, $k\geqslant 1$):

(i) $\vee$ is commutative and associative;

(ii) if $k<l$, then $\dfrac{k}{\{\mathbf G_1,\ldots,\mathbf G_k\}}\vee\dfrac{l}{\{\mathbf G_1,\ldots,\mathbf G_l\}}=\dfrac{k}{\{\mathbf G_1,\ldots,\mathbf G_k\}}$;

(iii) $\dfrac{k}{\{\mathbf G_1,\ldots,\mathbf G_m\}}=\bigvee\limits_{1\leqslant i_1<\ldots<i_k\leqslant m}\dfrac{k}{\{\mathbf G_{i_1},\ldots,\mathbf G_{i_k}\}}$.

Property (iii), in particularly, implies the equality $\dfrac{1}{\{\mathbf G_1,\ldots,\mathbf G_k\}}=\bigvee\limits^k_{i=1}\dfrac{1}{\{\mathbf G_i\}}$.

We now consider the set-theoretic intersection operation on fractoconvexities.

\begin{proposition} \label{prop1}
For any convexities $\mathbf G_1,\ldots,\mathbf G_k$, $k<\infty$, we have
\begin{gather*}
\dfrac{k}{\{\mathbf G_1,\ldots,\mathbf G_k\}}=\dfrac{1}{\{\mathbf G_1\cap\ldots\cap\mathbf G_k\}}=\dfrac{1}{\{\mathbf G_1\}}\cap\ldots\cap\dfrac{1}{\{\mathbf G_k\}}.
\end{gather*}
\end{proposition}
\begin{proof}
Without loss of generality, we put $k = 2$. We shall show that the convexity $\mathbf G_1\cap\mathbf G_2$ is $n$-ary. Indeed, let $\tilde\gop$ be the convex hull associated with $\mathbf G_1\cap\mathbf G_2$. From the properties of the convex hull operator it follows that
\begin{gather*}
A\in\mathbf G_1\cap\mathbf G_2\Rightarrow\forall B\subset A\bigl(|B|\leqslant n\Rightarrow\tilde\gop B\subset A\bigr).
\end{gather*}

On the other hand, if $\forall B\subset A\bigl(|B|\leqslant n\Rightarrow\tilde\gop B\subset A\bigr)$, then, in view of the inclusions $\gop_1 B,\gop_2 B\subset\tilde\gop B$, we have
\begin{gather*}
\forall B\subset A\bigl(|B|\leqslant n\Rightarrow\gop_1 B\subset A,\gop_2 B\subset A\bigr).
\end{gather*}
Since the convexities $\mathbf G_1$ and $\mathbf G_2$ is $n$-ary, we obtain $A\in\mathbf G_1\cap\mathbf G_2$, whence $\mathbf G_1\cap\mathbf G_2$ is $n$-ary.

Now, in our notation, we can write $\mathbf G_1\cap\mathbf G_2=\dfrac{1}{\{\mathbf G_1\cap\mathbf G_2\}}$. Since the convexities $\mathbf G_1$ and $\mathbf G_2$ are $n$-ary, the convexity $\mathbf G_1\cap\mathbf G_2$ coincides with the set $\dfrac{1}{\{\mathbf G_1\}}\cap\dfrac{1}{\{\mathbf G_2\}}$. Hence, the second required equality has been proved. In the same way, we obtain $A\in\dfrac{2}{\{\mathbf G_1,\mathbf G_2\}}\Leftrightarrow A\in\mathbf G_1\cap\mathbf G_2$, whence the first equality is true.
\end{proof}

Thus, if the cardinality in the numerator equals the cardinality of the denominator and $\mathbf G_i$ are $n$-ary for all $i=1,\ldots,k$, $k<\infty$, then the fractoconvexity $\dfrac{k}{\{\mathbf G_1,\ldots,\mathbf G_k\}}$ is the $n$-ary convexity $\mathbf G_1\cap\ldots\cap\mathbf G_k$. In other words, as in the case of biconvex sets, the collection of all $n$-ary multiconvex sets is an $n$-ary convexity.

\begin{corollary} \label{cor1}
For any fractoconvexity $\mathbf F=\bigvee\limits_i\dfrac{m_i}{\{\mathbf G_\lambda, \lambda\in M_i\}}\in\mathfrak F^{(\,n)}_{\mathrm{fin}}(X)$ we have
\begin{gather*}
\mathbf F=\bigvee\limits_{\substack{i;\\\Omega\subset M_i,|\Omega|=m_i}}\bigcap\limits_{\lambda_i\in\Omega}\dfrac{1}{\{\mathbf G_{\lambda_i}\}}.
\end{gather*}
\end{corollary}
\begin{proof}
Evidently, this equality follows from (iii) and Proposition~\ref{prop1}.
\end{proof}

\begin{proposition} \label{prop2}
The operations $\vee$ and $\cap$ satisfy the distributive laws on $\mathfrak F^{(\,n)}_{\mathrm{fin}}(X)$: $\forall\mathbf F_1,\mathbf F_2,\mathbf F_3\in\mathfrak F^{(\,n)}_{\mathrm{fin}}(X)$
\begin{gather*}
(\mathbf F_1\vee\mathbf F_2)\cap\mathbf F_3=(\mathbf F_1\cap\mathbf F_3)\vee(\mathbf F_2\cap\mathbf F_3),\\(\mathbf F_1\cap\mathbf F_2)\vee\mathbf F_3=(\mathbf F_1\vee\mathbf F_3)\cap(\mathbf F_2\vee\mathbf F_3).
\end{gather*}
\end{proposition}
\begin{proof}
These laws follow simply from (i)--(iii), Proposition~\ref{prop1}, and the distributivity of the operations $\vee$ and $\cap$ for the fractoconvexities $\dfrac{1}{\{\mathbf G_\lambda,\lambda\in\Lambda\}}\in\mathfrak F^{(\,n)}_{\mathrm{fin}}(X)$. The last property follows from the definition of $\vee$. For example, we show the distributivity of $\cap$ over $\vee$.
\begin{align*}
&A\in\left(\dfrac{1}{\{\mathbf G_\lambda,\lambda\in\Lambda_1\}}\vee\dfrac{1}{\{\mathbf G_\lambda,\lambda\in\Lambda_2\}}\right)\cap\dfrac{1}{\{\mathbf G_\lambda,\lambda\in\Lambda_3\}}\Leftrightarrow\\
&\Leftrightarrow\forall B\subset A\Bigl(|B|\leqslant n\Rightarrow\bigl(\bigl((\exists\lambda_1\in\Lambda_1: \gop_{\lambda_1} B\subset A)\vee(\exists\lambda_2\in\Lambda_2: \gop_{\lambda_2} B\subset A)\bigr)\wedge(\exists\lambda_3\in\Lambda_3: \gop_{\lambda_3} B\subset A)\bigr)\Bigr)\Leftrightarrow\\
&\Leftrightarrow\forall B\subset A\Bigl(|B|\leqslant n\Rightarrow\bigl((\exists\lambda_1\in\Lambda_1\exists\lambda_3\in\Lambda_3: \gop_{\lambda_1} B\cup\gop_{\lambda_3} B\subset A)\vee(\exists\lambda_2\in\Lambda_2\exists\lambda_3\in\Lambda_3:\\
&\gop_{\lambda_2} B\cup\gop_{\lambda_3} B\subset A)\bigr)\Bigr)\Leftrightarrow A\in\left(\dfrac{1}{\{\mathbf G_{\lambda},\lambda\in\Lambda_1\}}\cap\dfrac{1}{\{\mathbf G_{\lambda},\lambda\in\Lambda_3\}}\right)\vee\left(\dfrac{1}{\{\mathbf G_{\lambda},\lambda\in\Lambda_2\}}\cap\dfrac{1}{\{\mathbf G_{\lambda},\lambda\in\Lambda_3\}}\right).
\end{align*}
\end{proof}

In the same way, one can prove the absorption law:
\begin{gather*}
(\mathbf F_1\vee\mathbf F_2)\cap\mathbf F_1=\mathbf F_1,\quad (\mathbf F_1\cap\mathbf F_2)\vee\mathbf F_1=\mathbf F_1.
\end{gather*}

Let $\hat{\mathfrak F}^{(\,n)}_\mathrm{fin}(X)$ be some subfamily of $\mathfrak F^{(\,n)}_\mathrm{fin}(X)$ which is closed under a finite number of applications of the operations $\vee$ and $\cap$. Obviously, from the properties obtained above for these operations it follows that $\hat{\mathfrak F}^{(\,n)}_\mathrm{fin}(X)$ is a distributive lattice.

\section{Independent convexities}
Convexities $\mathbf G_\lambda\in\mathfrak G^{(\,n)}(X)$, $\lambda\in\Lambda$ will be called {\itshape mutually $(n)$-independent} or, briefly, {\itshape independent}, if the following condition holds
\begin{gather*}
\forall A\in(n)\text{-}\dfrac{1}{\{\mathbf G_\lambda,\lambda\in\Lambda\}}: A=\bigcap\limits_{\lambda\in\Lambda}\gop_\lambda A.
\end{gather*}

If the equality does not necessarily hold for all $A\in(n)\text{-}\dfrac{1}{\{\mathbf G_\lambda,\lambda\in\Lambda\}}$, then the maximal subfamily of $(n)\text{-}\dfrac{1}{\{\mathbf G_\lambda,\lambda\in\Lambda\}}$ for which the equality holds will be called the {\itshape $(n)$-independence domain} of the convexities $\mathbf G_\lambda,\lambda\in\Lambda$, and will be denoted by $(n)\text{-}\mathrm{idc}(\mathbf G_\lambda,\lambda\in\Lambda)$. The sets belonging to $(n)\text{-}\mathrm{idc}(\mathbf G_\lambda,\lambda\in\Lambda)$ will be called the {\itshape elements of $(n)$-independence} of the convexities $\mathbf G_\lambda,\lambda\in\Lambda$. In what follows, if not otherwise stated, the prefix $(n)$- will be omitted.

Given a set $M\subset\dfrac{1}{\{\mathbf G_\lambda,\lambda\in\Lambda\}}$ we consider the question whether the inclusion $M\subset\mathrm{idc}(\mathbf G_\lambda,\lambda\in\Lambda)$ is true. This question can be answered in two ways. First, the condition of convexities independence is directly verified for every $A\in M$. However, as was shown by examples in~\cite{4}, this way can be very laborious. Second, the required inclusion may be checked by using some uncomplicated sufficient condition of convexities independence. One of such conditions will be given by us in Proposition~\ref{prop3}.

The importance of the notion of independent convexities is that, in special cases, we can obtain statements about the separation property for two certain elements of independence of convexities, and the statements are analogous to the separation theorems for two convex sets. Moreover, these elements are separated by a set represented by intersection of some $\gop_\lambda$\--halfspaces. A recent investigation on this topic for semiconvex sets in $S^2$ has been carried out in \cite{4} by the author.

Now we shall give the following definition, which will play an important role in Lemma~\ref{lem1} and in Proposition~\ref{prop3}.

Two finitely defined (not necessarily $n$-ary) convexities $\mathbf G_1$ and $\mathbf G_2$ on $X$, $|X|\geqslant 4$, are said to be {\itshape conically independent} provided the following condition is true
\begin{gather*}
\forall n\geqslant 4\:\forall B\subset X, |B|=n-1\:\forall x_n\in X\backslash B\;\forall x\in\gop^{\cap}\!\left(B\cup\{x_n\}\right)\exists y_1,y_2\in\gop^{\cap}\!B:\\x\in\gop_1\{y_1,x_n\}\cap\gop_2\{y_2,x_n\}.
\end{gather*}
(Here and in what follows, we use the notation: $\gop^{\cap}\!A=\gop_1 A\cap\gop_2 A$.)

\begin{lemma} \label{lem1}
Let the convexities $\mathbf G_1$ and $\mathbf G_2$ be conically independent, and let the set $A$ satisfy the condition
\begin{equation}\label{3.1}
\forall x_1,x_2,x_3\in A: \gop^{\cap}\{x_1,x_2,x_3\}\subset A.
\end{equation}
Then we have $A=\gop^{\cap}\!A$.
\end{lemma}
\begin{proof}
Since the convexities $\mathbf G_1$ and $\mathbf G_2$ are finitely defined, we see that
\begin{equation}\label{3.2}
\forall x\in\gop^{\cap}\!A\,\exists x_1,\ldots,x_n\in A: x\in\gop^{\cap}\{x_1,\ldots,x_n\}.
\end{equation}

From (\ref{3.1}) it follows that the proof is trivial if $n<4$; therefore, we may put $n\geqslant 4$. Iterating the conical independence of $\mathbf G_1$ and $\mathbf G_2$ in (\ref{3.2}), we get the chain of implications:
\begin{equation}\label{3.3}
\begin{gathered}
\exists y_1,y_2\in\gop^{\cap}\{x_1,\ldots,x_{n-1}\}: x\in\gop_1\{y_1,x_n\}\cap\gop_2\{y_2,x_n\}\Rightarrow\\
\Rightarrow\exists y_{11},y_{12},y_{21},y_{22}\in\gop^{\cap}\{x_1,\ldots,x_{n-2}\}: y_1\in\gop_1\{y_{11},x_{n-1}\}\cap\gop_2\{y_{12},x_{n-1}\},\\
y_2\in\gop_1\{y_{21},x_{n-1}\}\cap\gop_2\{y_{22},x_{n-1}\}\Rightarrow\ldots\Rightarrow \exists y_{\underbrace{\scriptstyle{1\ldots 11}}_{n-4}},y_{\underbrace{\scriptstyle{1\ldots 12}}_{n-4}},\ldots,y_{\underbrace{\scriptstyle{2\ldots 21}}_{n-4}},\\
y_{\underbrace{\scriptstyle{2\ldots 22}}_{n-4}}\in\gop^{\cap}\{x_1,\ldots,x_4\}: y_{\underbrace{\scriptstyle{1\ldots 11}}_{n-5}}\in\gop_1\{y_{\underbrace{\scriptstyle{1\ldots 11}}_{n-4}},x_5\}\cap\gop_2\{y_{\underbrace{\scriptstyle{1\ldots 12}}_{n-4}},x_5\},\ldots\Rightarrow\\
\Rightarrow\exists y_{\underbrace{\scriptstyle{1\ldots 11}}_{n-3}},y_{\underbrace{\scriptstyle{1\ldots 12}}_{n-3}},\ldots\in\gop^{\cap}\{x_1,x_2,x_3\}: y_{\underbrace{\scriptstyle{1\ldots 11}}_{n-4}}\in\gop_1\{y_{\underbrace{\scriptstyle{1\ldots 11}}_{n-3}},x_4\}\cap\gop_2\{y_{\underbrace{\scriptstyle{1\ldots 12}}_{n-3}},x_4\}\ldots
\end{gathered}
\end{equation}

Without loss of generality, consider the points $u=y_{\underbrace{\scriptstyle{1\ldots 11}}_{n-4}}$, $v_1=y_{\underbrace{\scriptstyle{1\ldots 11}}_{n-3}}$, and $v_2=y_{\underbrace{\scriptstyle{1\ldots 12}}_{n-3}}$. From (\ref{3.1}) it follows that $x_4,v_1,v_2\in A$; hence, again from (\ref{3.1}), $\gop^{\cap}\{x_4,v_1,v_2\}\subset A$. Obviously, if $u\in\{x_4,v_1,v_2\}$, then $u\in A$. Let $u\notin\{x_4,v_1,v_2\}$. Taking into account the properties of a convex hull operator, for each $i\in\{1,2\}$, we obtain
\begin{gather*}
\gop_1\{v_1,x_4\}\cap\gop_2\{v_2,x_4\}\subset\gop_i\{v_i,x_4\}\subset\gop_i\{x_4,v_1,v_2\}.
\end{gather*}
whence $\gop_1\{v_1,x_4\}\cap\gop_2\{v_2,x_4\}\subset\gop^{\cap}\{x_4,v_1,v_2\}$. Since $u\in\gop_1\{v_1,x_4\}\cap\gop_2\{v_2,x_4\}$, we see again that $u\in A$.

Moving back along the chain of implications (\ref{3.3}) in which the previous procedure applies to all points $y_{\displaystyle\cdot\!\cdot\!\cdot}$ and $x_{\displaystyle\cdot}$, we obtain $\gop_1\{y_1,x_n\}\cap\gop_2\{y_2,x_n\}\subset\gop^{\cap}\{x_n,y_1,y_2\}\subset A$, whence $x\in A$. From the arbitrariness of $x$, we conclude that $\gop_1 A\cap\gop_2 A\subset A$. The reverse inclusion is evident.
\end{proof}

\begin{proposition} \label{prop3}
Let the convexities $\mathbf G_1,\mathbf G_2\in\mathfrak G^{(\,3)}(X)$ be conically independent; then $\mathbf G_1$ and $\mathbf G_2$ are \linebreak independent.
\end{proposition}
\begin{proof}
Consider an arbitrary set $A\in(3)\text{-}\dfrac{1}{\{\mathbf G_1,\mathbf G_2\}}$. The 3\--arity of $\mathbf G_1$ and $\mathbf G_2$ imply that for any $x_1,x_2,x_3\in A$ there exists an $i\in\{1,2\}$ for which $\gop_i\{x_1,x_2,x_3\}\subset A$. Since $\gop^{\cap}\{x_1,x_2,x_3\}\subset\gop_i\{x_1,x_2,x_3\}$, it follows that $A$ satisfy (\ref{3.1}). Thus, by Lemma~\ref{lem1} we have $A=\gop^{\cap}A$.
\end{proof}
\section{Examples of spaces with fractoconvexity}
To describe the fractoconvexities in Examples 1--3, we need the following construction \cite{4}.

Let $S$ be the 2-shpere in $\mathbb R^3$ with center at the origin, $B$ be the closed ball bounded by $S$, $C$ be an arbitrary fixed set in the interior of $B$, and the symbol $[,]$ denotes the line segment operator in $\mathbb R^3$. The convexities $\mathbf G(c), c\in C$ on $S$ are defined analogously to the convexity in the sense of Robinson \cite{5} but with respect to the points $c\in C$, respectively. This means that a set $A\subset S$ is $\mathbf G(c)$-convex iff, for two distinct points $x_1,x_2$ such that the straight line determined by them does not pass through the point $c$, the set cut out by the 2-dimensional cone with vertex $c$ and base $[x_1,x_2]$ is entirely contained in $A$. It is readily seen that all convexities $\mathbf G(c), c\in C$ are binary, and the segment $\gop_c\{x_1,x_2\}$ joining two points $x_1,x_2\in S$ coincides with the subset of $S$ mentioned above if $x_1,x_2$, and $c$ are non-collinear and equals $\{x_1,x_2\}$ otherwise.

{\bfseries\itshape Example 1.}

Let $C=\{c_0,c_1\}$; then the fractoconvexity $\mathbf F_1=(2)\text{-}\dfrac{1}{\{\mathbf G(c_0),\mathbf G(c_1)\}}$ is the family of all 2\--semiconvex sets with respect to $\mathbf G(c_0)$ and $\mathbf G(c_1)$. Among these 2\--semiconvex sets, the so-called $\gop_{01}$-regular 2\--semiconvex sets are very important. (A set $A$ is called $\gop_{01}$\--regular if there exists an open halfspace $H$ in $\mathbb R^3$ such that $c_0,c_1\in H$ and $A\subset H^c$.) It has been shown in \cite{4} that every  $\gop_{01}$-regular 2-semiconvex set belongs to the independence domain of $\mathbf G(c_0)$ and $\mathbf G(c_1)$.

{\bfseries\itshape Example 2.}

Let $C=\{c_\lambda:=\lambda c_0+(1-\lambda)c_1\mid\lambda\in[0,1]\}$ and let us introduce the convexities $\mathbf G'(c_\lambda)$, $c_\lambda\in C$, which are the restriction of $\mathbf G(c_\lambda)$, $c_\lambda\in C$ to the subfamilies of the sets that are $\gop_{01}$-regular with respect to some fixed open halfspace $H$ parallel to the segment $[c_0,c_1]$. Considering the fractoconvexity $\mathbf F_2=(2)\text{-}\dfrac{1}{\{\mathbf G'(c_\lambda),c_\lambda\in C\}}$, by analogy with the previous example, one can prove that the convexities $\mathbf G'(c_\lambda)$, $c_\lambda\in C$ are independent.

{\bfseries\itshape Example 3.}

Let $C$ be as in Example~2. Consider the multiconvexities
\begin{equation*}
\mathbf G_1=(2)\text{-}\left\{\mathbf G'(c_\lambda)\mid c_\lambda\in[c_0,(c_0+c_1)/2]\right\},
\mathbf G_2=(2)\text{-}\left\{\mathbf G'(c_\lambda)\mid c_\lambda\in[(c_0+c_1)/2,c_1]\right\},
\end{equation*}
and put $\mathbf F_3=(2)\text{-}\dfrac{1}{\{\mathbf G_1,\mathbf G_2\}}$. Obviously, these multiconvexities are binary. By definition of an $n$-ary convexity, they are 3-ary as well.

Consider the convexity $\mathbf G'((c_0+c_1)/2)$. By $\tilde\gop$ denote the corresponding convex hull operator. It is easily seen that
\begin{equation}\label{4.1}
\forall n\geqslant 2\,\forall x_1,\ldots,x_n\in X, x_i\ne x_j (i\ne j)\,\forall x\in\tilde\gop\{x_1,\ldots,x_n\}\exists y\in\tilde\gop\{x_1,\ldots,x_{n-1}\}:x\in\tilde\gop\{y,x_n\}.
\end{equation}
Moreover, the operators $\gop_1$, $\gop_2$, and $\tilde\gop$ are related to each other by the condition
\begin{equation}\label{4.2}
\forall n\in\mathbb N\,\forall x_1,\ldots,x_n\in X\,\gop^\cap\{x_1,\ldots,x_n\}=\tilde\gop\{x_1,\ldots,x_n\}.
\end{equation}

Combining (\ref{4.1}) and (\ref{4.2}), we see that the 3-ary convexities $\mathbf G_1$ and $\mathbf G_2$ are conically independent; therefore, considering the fractoconvexity $\mathbf F_3=(2)\text{-}\dfrac{1}{\{\mathbf G_1,\mathbf G_2\}}$ and using Proposition~\ref{prop3}, we conclude that these convexities are independent.

{\bfseries\itshape Example 4.}

Suppose that the binary convexities $\mathbf G_1$ and $\mathbf G_2$ and the fractoconvexity $\mathbf F_4$ are defined on $\mathbb Z$ as follows. Let $\mathbf G_1$ be the collection of all sets $A\cap\mathbb Z$, where $A\subset\mathbb R$ is a standard convex set. Suppose $f$ is a bijective function from $\mathbb Z$ to itself. Using $f$, for any $x_1,x_2\in\mathbb Z$ we define the segment $\gop_2\{x_1,x_2\}$ by the formula
\begin{gather*}
\gop_2\{x_1,x_2\}:=f\bigl(\gop_1\{f^{-1}(x_1),f^{-1}(x_2)\}\bigr).
\end{gather*}
We put $\mathbf G_2=\left\{A\subset\mathbb Z\mid\forall x_1,x_2\in A\,\gop_2\{x_1,x_2\}\subset A\right\}$.

If a set $A\subset\mathbb Z$ is bounded, then the operators $\gop_1$ and $\gop_2$ satisfy the following equalities:

1) $\gop_1 A=\mathbb Z\cap\{a_1\lambda+b_1(1-\lambda)\mid\lambda\in[0,1],a_1=\min A,b_1=\max A\}$;

2) $\gop_2 A:=f\bigl(\gop_1\{a_2,b_2\}\bigr)=\gop_2\bigl\{f(a_2),f(b_2)\bigr\}$, where $a_2=\min\limits_{x\in A}f^{-1}(x)$, $b_2=\max\limits_{x\in A}f^{-1}(x)$.

Consider the fractoconvexity $\mathbf F_4=(2)\text{-}\dfrac{1}{\{\mathbf G_1,\mathbf G_2\}}$. The following proposition is valid.

\begin{proposition} \label{prop4}
If a set $A\in\mathbf F_4$ is bounded, then $A\in\mathrm{idc}(\mathbf G_1,\mathbf G_2)$.
\end{proposition}
\begin{proof}
Since $A\subset\gop_1 A\cap\gop_2 A$, it suffices to verify the reverse inclusion. If $\gop_1 A\subset A$ or $\gop_2 A\subset A$, then proof is trivial. Therefore, we assume that $\gop_1 A\not\subset A$ and $\gop_2 A\not\subset A$.

From invertibility of the function $f$ and from the definition of the points $a_2$ and $b_2$ it follows that the points $f(a_2)$ and $f(b_2)$ belong to $A$. The set $A$ is semiconvex; hence, from the relations $\gop_1\{a_1,b_1\}=\gop_1 A\not\subset A$ and $\gop_2\bigl\{f(a_2),f(b_2)\bigr\}=\gop_2 A\not\subset A$, we obtain
\begin{equation}\label{4.3}
\begin{gathered}
\gop_2\{a_1,b_1\}=f\bigl(\gop_1\{f^{-1}(a_1),f^{-1}(b_1)\}\bigr)\subset A,\\
\gop_1\bigl\{f(a_2),f(b_2)\bigr\}\subset A.
\end{gathered}
\end{equation}

Suppose that the set $(\gop_1 A\cap\gop_2 A)\setminus A$ is nonempty, that is,
\begin{equation}\label{4.4}
\exists x\in\bigl(\gop_1\{a_1,b_1\}\cap\gop_2\bigl\{f(a_2),f(b_2)\bigr\}\bigr)\setminus A.
\end{equation}
Under this assumption and by (\ref{4.3}), we see that $x\notin\gop_1\bigl\{f(a_2),f(b_2)\bigr\}\subset A$. But from (\ref{4.4}) it follows that $x\in\gop_1\{a_1,b_1\}$; hence, we have the following possible locations of the point $x$: either $x\in\gop_1\{a_1,f(\cdot)\}$ or $x\in\gop_1\{f(\cdot),b_1\}$, where $f(\cdot)$ belongs to $\{f(a_2),f(b_2)\}$ and is determined depending on the  mutual position of $f(a_2)$ and $f(b_2)$ (see below). Both these cases are investigated equally. For example, consider only the first of them. We have two subcases:
\begin{gather*}
\text{either }x\in\gop_1\{a_1,f(a_2)\} \text{ if } f(a_2)\leqslant f(b_2) \text{; or }x\in\gop_1\{a_1,f(b_2)\} \text{ if } f(a_2)>f(b_2).
\end{gather*}
(Here we have taken into account that $f(a_2)=f\bigl(\min\limits_{x\in A} f^{-1}(x)\bigr)\geqslant\min A=a_1$.) Without loss of generality it can be investigated one subcase, for example, $f(a_2)\leqslant f(b_2)$.

Since $x\notin A$ and $x\in\gop_1\{a_1,f(a_2)\}$, we obtain
\begin{equation}\label{4.5}
\gop_1\{a_1,f(a_2)\}\not\subset A.
\end{equation}
Therefore, we have $\gop_2\{a_1,f(a_2)\}\subset A$, since $A$ is semiconvex. Hence, by the definition of $\gop_2$, we get
\begin{equation}\label{4.6}
f\bigl(\gop_1\{a_2,f^{-1}(a_1)\}\bigr)\subset A.
\end{equation}

From (\ref{4.4}) it follows that $x\in\gop_2\bigl\{f(a_2),f(b_2)\bigr\}$, in other words, that $x\in f(\gop_1\{a_2,b_2\})$. The definitions of the points $a_2$ and $b_2$ imply that $f^{-1}(a_1)\in\gop_1\{a_2,b_2\}$. Combining these remarks with (\ref{4.6}) and $x\notin A$, we obtain $x\in f\bigl(\gop_1\{f^{-1}(a_1),b_2\}\bigr)\not\subset A$. At the same time, $f\bigl(\gop_1\{f^{-1}(a_1),b_2\}\bigr)=\gop_2\bigl\{a_1,f(b_2)\bigr\}$; hence, $\gop_1\{a_1,f(b_2)\}\subset A$, since $A$ is semiconvex.

By virtue of the assumption $f(a_2)\leqslant f(b_2)$, we have $\gop_1\{a_1,f(a_2)\}\subset\gop_1\{a_1,f(b_2)\}$, whence $\gop_1\{a_1,f(a_2)\}\subset A$. But the last contradicts inclusion (\ref{4.5}). Thus, $(\gop_1 A\cap\gop_2 A)\setminus A=\varnothing$, and since $A\subset\gop_1 A\cap\gop_2 A$, the last equality is true iff $A=\gop_1 A\cap\gop_2 A$.
\end{proof}


\end{document}